\documentclass[a4paper,11pt]{amsart}

\title{Reverse superposition estimates in Sobolev spaces}

\author{Jean Van Schaftingen}
\address{Universit\'e catholique de Louvain\\ 
Institut de Recherche en Math\'ematique et Physique\\
Chemin du Cyclotron 2 bte L7.01.01\\
1348 Louvain-la-Neuve\\
Belgium}
\email{Jean.VanSchaftingen@UCLouvain.be}

\usepackage{enumerate}

\usepackage[T1]{fontenc}
\usepackage[utf8]{inputenc}
\usepackage{geometry}
\usepackage{lmodern} 
\usepackage{microtype}

\usepackage{amssymb}
\usepackage{amsmath}
\usepackage{amsthm}
\usepackage{esint}

\usepackage{constants}
\usepackage{paralist}
\usepackage{mathrsfs}

\usepackage{todonotes}

\usepackage[pdftitle={Reverse superposition estimates in Sobolev spaces},%
hidelinks,
pdfauthor={Jean Van Schaftingen},%
final]{hyperref}

\usepackage[nameinlink%
]{cleveref}

\usepackage[abbrev]{amsrefs}

\newtheorem{theorem}{Theorem}[section]
\newtheorem{proposition}[theorem]{Proposition}

\theoremstyle{definition}

\theoremstyle{remark}
\newtheorem{remark}[theorem]{Remark}

\numberwithin{equation}{section}

\usepackage{mathtools}

\DeclarePairedDelimiter{\brk}{(}{)}
\DeclarePairedDelimiter{\abs}{\lvert}{\rvert}

\DeclarePairedDelimiterX{\intvc}[2]{[}{]}{#1,#2}
\DeclarePairedDelimiterX{\intvl}[2]{(}{]}{#1,#2}
\DeclarePairedDelimiterX{\intvr}[2]{[}{)}{#1,#2}
\DeclarePairedDelimiterX{\intvo}[2]{(}{)}{#1,#2}
\DeclarePairedDelimiterX{\setcond}[2]{\{}{\}}{#1 \,\delimsize\vert\, #2}

\newcommand{\defeq}{\coloneqq}

\newcommand{\compose}{\,\circ\,}
\DeclareMathOperator*{\essosc}{ess\,osc}
\DeclareMathOperator*{\esssup}{ess\,sup}
\DeclareMathOperator*{\essrg}{ess\,rg}
\DeclareMathOperator{\diam}{diam}

\newcommand{\Rset}{\mathbb{R}}
\newcommand{\Nset}{\mathbb{N}}
\newcommand{\Zset}{\mathbb{Z}}

\newcommand{\dif}{\,\mathrm{d}}
\DeclareMathOperator{\sign}{sgn}

\newcommand{\eofs}{\,}

\begin{document}
\begin{abstract}
We study when and how the norm of a function $u$ in the homogeneous Sobolev spaces $\dot{W}^{s, p} (\mathbb{R}^n, \mathbb{R}^m)$, with $p \ge 1$ and either $s = 1$ or $s > 1/p$, is controlled by the norm of composite function $f \circ u$ in the same space.
\end{abstract}

\maketitle

\section{Introduction}
The absolute value preserves weak differentiability despite its non-differentiability at \(0\)  \cite{Marcus_Mizel_1972} (see also \citelist{\cite{Willem_2013}*{corollary 6.1.14}\cite{Ziemer_1989}*{corollary 2.1.8}\cite{Gilbarg_Trudinger_1983}*{lemma 7.6}}). 
More precisely, if \(u\) belongs to the homogeneous first-order Sobolev space \(\dot{W}^{1, p} (\Omega, \Rset)\) for some \(p \in [1, \infty)\), that is, if the function \(u \colon \Omega \to \Rset\) is weakly differentiable on the open set \(\Omega \subseteq \Rset^n\) and its weak derivative \(D u\) satisfies the integrability condition \(\int_{\Omega} \abs{D u}^p < +\infty\), then \(\abs{u} \in \dot{W}^{1, p} (\Omega, \Rset)\); moreover, one has then
\begin{equation}
\label{eq_aihohwish3ahl1eiTheixaeh}
 D \abs{u} = \sign (u) Du \qquad \text{almost everywhere in \(\Omega\)},
\end{equation}
where the signum function \(\sign\) is defined by \(\sign (t) = -1\) when \(t < 0\), \(\sign(0) = 0\) and \(\sign(t)=1\) when \(t > 0\). A consequence of the identity \eqref{eq_aihohwish3ahl1eiTheixaeh} and of the fact that \(Du = 0\) almost everywhere on \(u^{-1} (\{0\})\) is the  integral identity
\begin{equation}
 \int_{\Omega} \abs{D \abs{u}}^p = \int_{\Omega} \abs{D u}^p,
\end{equation}
which can be interpreted either as an estimate in the homogeneous Sobolev space \(\dot{W}^{1, p} (\Omega)\) for \(\abs{u}\) in terms of \(u\), or conversely as an a priori estimate for \(u\) in terms of \(\abs{u}\), \emph{provided} it is known a priori that \(u \in \dot{W}^{1, p} (\Omega, \Rset)\). We will adopt the latter point of view. 

This result about the absolute value is a particular case of reverse estimates for superposition operators \(u \mapsto f \compose u\),  for \(u \in \dot{W}^{1, p} (\Rset^n, \Rset^m)\) and \(f : \Rset^m \to \Rset^\ell\). 
We state in \cref{theorem_reverse_first} below a wide condition on the function \(f\) which ensures that \(u \in \dot{W}^{1, p} (\Rset^n, \Rset^m)\) is controlled by \(f \compose u \in \dot{W}^{1, p} (\Rset^n, \Rset^\ell)\); this condition does not require that \(f \compose u \in \dot{W}^{1, p} (\Rset^n, \Rset^\ell)\) when \(u \in \dot{W}^{1, p} (\Rset^n, \Rset^m)\).

We next consider the question whether such reverse superposition estimate extend to the \emph{homogeneous fractional Sobolev space}  
\begin{equation}
\dot{W}^{s, p} (\Omega, \Rset^m)
\defeq 
\setcond[\bigg]{u\colon \Omega \to \Rset^m}{\iint\limits_{\Omega \times \Omega} \frac{\abs{u (y)- u(x)}^p}{\abs{y - x}^{n + sp}}  \dif y \dif x < +\infty},
\end{equation}
with \(0 < s < 1\) and \(1 \le p < +\infty\).
Athough there is no identity such as \eqref{eq_aihohwish3ahl1eiTheixaeh} for fractional Sobolev spaces, we prove that when \(sp > 1\) there exists a constant such that for every \(u \in \dot{W}^{s, p} (\Omega, \Rset)\), the reverse estimate 
\begin{equation}
\label{eq_ait6beethei2chooCah7zohz}
  \iint\limits_{\Omega \times \Omega} \frac{\abs{u (y) - u (x)}^p}{\abs{y - x}^{n + sp}} \dif y \dif x\\
 \le C
 \iint\limits_{\Omega \times \Omega} \frac{\abs[\big]{\abs{u (y)} - \abs{u (x)}}^p}{\abs{y - x}^{n + sp}} \dif y \dif x
\end{equation}
holds. 
The estimate \eqref{eq_ait6beethei2chooCah7zohz} is a particular case of a class of reverse estimates for superposition operators (\cref{theorem_reverse_fractional}). The proof of \eqref{eq_ait6beethei2chooCah7zohz} is based on a reverse oscillation obtained by Petru Mironescu and the author in the lifting of fractional Sobolev mappings over a compact covering \cite{Mironescu_VanSchaftingen_CpctLift}.

When \(sp \le 1\), the reverse estimate \eqref{eq_ait6beethei2chooCah7zohz} fails. 
In fact there exists an unbounded sequence \((u_j)_{j \in \Nset}\) in \(\dot{W}^{s, p} (\Omega, \Rset)\) such that for each \(j \in \Nset\) the function \(\abs{u_j}\) is constant on \(\Omega\) (\(sp <1\), \cref{proposition_non_estimate_sp_lt_1}) or such that the sequence \((\abs{u_j})_{j \in \Nset}\) remains bounded in \(\dot{W}^{s, p} (\Omega, \Rset)\) (\(sp = 1\), \cref{proposition_non_estimate_sp_eq_1}).

When \(p = 2\), an estimate of the form \eqref{eq_ait6beethei2chooCah7zohz} still holds when \(1 < s < 3/2\) with a suitable definition of fractional Sobolev norm \cite{Musina_Nazarov_2019}.

\section{Reverse estimates for first-order Sobolev spaces}
Our first result is a reverse estimate for weakly differentiable functions.
\begin{theorem}
\label{theorem_reverse_first}
If the set \(\Omega \subseteq \Rset^n\) is open, if the function \(f : \Rset^m \to \Rset^\ell\) is Borel-measurable, if \(u \in \dot{W}^{1,1}_{\mathrm{loc}} (\Omega, \Rset^m)\) and if \(f \compose u \in \dot{W}^{1, 1}_{\mathrm{loc}} (\Omega, \Rset^\ell)\), then for almost every \(x \in \Omega\) and every \(h\in \Rset^n\),
\begin{equation}
\label{eq_ilie8Ob4neetue7ahtoh9ub2}
 \abs{Du(x)[h]} \le \abs{D (f \compose u)(x)[h]} \limsup_{y \to u (x)} \frac{\abs{y - u (x)}}{\abs{f (y) - f (u (x))}}
\end{equation}
\end{theorem}

\begin{remark}
If the function \(f\) is classically differentiable at the point \(u (x)\), then 
\begin{equation}
 \limsup_{y \to u (x)} \frac{\abs{y - u (x)}}{\abs{f (y) - f (u (x))}}
 = \frac{1}{\sup\, \setcond[\big]{\abs{D f(u(x)) [k]}/\abs{k}}{k \in \Rset^m \setminus \{0\}}}.
\end{equation}
\end{remark}

\begin{remark}
If for each \(y \in \Rset\) the function \(f\) is defined as \(f (y) \defeq \abs{y}\), then we have for every \(z \in \Rset\),
\begin{equation}
  \limsup_{y \to z} \frac{\abs{y - z}}{\abs{\abs{y} - \abs{z}}} = 1,
\end{equation}
and \eqref{eq_ilie8Ob4neetue7ahtoh9ub2} is then in this particular case a consequence of \eqref{eq_aihohwish3ahl1eiTheixaeh}.
\end{remark}

The proof of \cref{theorem_reverse_first} follows the strategy of the general chain rule for weakly differentiable functions \cite{Ambrosio_DalMaso_1990}.

\begin{proof}[Proof of \cref{theorem_reverse_first} when \(n = 1\)]
By the characterisation of weakly differentiable functions on an interval (see for example \cite{Leoni_2017}*{theorem 7.13}),  for almost every \(x \in \Omega\) there exists a sequence \((h_j)_{j \in \Nset}\) in \(\Rset \setminus \{0\}\) converging to \(0\) such that both
\begin{gather}
\label{eq_jeisi1kadei0EeshooQuaiVi}
  \lim_{j \to \infty} \frac{u(x + h_j) - u (x)}{h_j} = u'(x)
  \intertext{and}
\label{eq_Nah1eCh9engee9ier8Que8ri}
  \lim_{j \to \infty} \frac{f(u(x + h_j)) - f(u (x))}{h_j} = (f \compose u)'(x).  
\end{gather}
Assuming without loss of generality that 
\[
 \limsup_{y \to u (x)} \frac{\abs{y - u (x)}}{\abs{f (y) - f (u (x))}}
 < +\infty
\]
we have for \(j\in \Nset\) large enough 
\(f (u (x) + h_j) \ne f (u (x))\); it then follows from the limits \eqref{eq_jeisi1kadei0EeshooQuaiVi} and \eqref{eq_Nah1eCh9engee9ier8Que8ri} that 
\begin{equation}
\begin{split}
 \abs{u'(x)}
 &= 
 \abs{(f \compose u)' (x)}
 \lim_{j \to \infty} \frac{\abs{u(x + h_j) - u (x)}}{\abs{f(u(x + h_j)) - f(u (x))}}\\
 &\le \abs{(f \compose u)' (x)} \limsup_{y \to u (x)} \frac{\abs{y - u (x)}}{\abs{f (y) - f (u (x))}} .\qedhere
\end{split}
\end{equation}
\end{proof}

\begin{proof}[Proof of \cref{theorem_reverse_first} when \(n \ge 2\)]
The proof goes by noting that the restrictions of \(u\) and \(f \compose u\) to almost every one-dimensional line \(L\) are weakly differentiable (see for example \cite{Leoni_2017}*{theorem 10.35}), applying the one-dimensional case and concluding by Fubini's theorem.
\end{proof}

\section{Fractional Sobolev spaces}
In the fractional case, we have the following counterpart of \cref{theorem_reverse_first}.

\begin{theorem}
\label{theorem_reverse_fractional}
For every \(s \in \intvo{0}{1}\) and \(p \in \intvr{1}{+\infty}\) satisfying \(sp > 1\), there exists a constant \(C\) such that for every convex set \(\Omega \subseteq \Rset^n\) and every \(f : \Rset^m \to \Rset^\ell\), if \(u \in \dot{W}^{s, p} (\Omega, \Rset^m)\) and if \(f \compose u \in \dot{W}^{s, p} (\Omega, \Rset^\ell)\), then 
\begin{multline}
\label{eq_xah0iechah6Iexaef8RoB5ae}
 \iint\limits_{\Omega \times \Omega} \frac{\abs{u (y) - u (x)}^p}{\abs{y - x}^{n + sp}} \dif y \dif x\\
 \le 
 C\,
 \brk[\Big]{
 \sup\, \setcond[\Big]{\frac{\diam (K)}{\diam (f(K))}}{K \subset \essrg_{\Rset^n} u \text{ compact, connected and } \diam(K) > 0}
}^p\\
\times 
 \iint\limits_{\Omega \times \Omega} \frac{\abs{f(u (y)) - f (u (x))}^p}{\abs{y - x}^{n + sp}} \dif y \dif x\eofs.
\end{multline}
\end{theorem}

Here \(\essrg u\)  denotes the \emph{essential range} of the function \(u : \Omega \to \Rset^m\) with respect to Lebesgue's \(n\)--dimensional measure \(\mathcal{L}^n\), defined as 
\begin{equation}
 \essrg_{\Rset^n} u \defeq \setcond[\Big]{y \in \Rset^m }{\text{for each \(\varepsilon  > 0,\) } \mathcal{L}^n \brk[\big]{u^{-1} (B_\varepsilon (y))} > 0}.
\end{equation}

Our main tool to prove \cref{theorem_reverse_fractional} is the following reverse oscillation inequality  \cite{Mironescu_VanSchaftingen_CpctLift}.

\begin{proposition}
  \label{proposition_reverse_oscillation}
  If the set \(\Omega \subseteq \Rset^n\) is convex and if \(sp > 1\),
  then there exists a constant \(C\) such that for every \(u \in \dot{W}^{s, p} (\Omega, \Rset^m)\) one has 
  \begin{equation}
   \iint\limits_{\Omega \times \Omega}\frac{(\essosc_{[x, y]} u)^p}{\abs{y - x}^{n + sp}} \dif y \dif x
    \le C\iint\limits_{\Omega \times \Omega} \frac{\abs{u (y)- u(x)}^p}{\abs{y - x}^{n + sp}}  \dif y \dif x
    \eofs.
  \end{equation}
\end{proposition}

Here we have defined the segment \([x, y] = \setcond{(1 -t)x + ty}{0 \le t \le 1}\) and the \emph{essential oscillation} 
\begin{equation}
\label{eq_qui5eihee1rai9IengieTeeN}
  \essosc_{[x, y]} u  \defeq
  \esssup_{t, r \in [0, 1]} \, \abs[\big]{u \brk[\big]{(1 - r)x + ry} - u \brk[\big]{(1 - t)x + ty}}
\end{equation}
\Cref{proposition_reverse_oscillation} is proved for \(n = 1\) and extended by Fubini-type arguments to higher dimension \cite{Mironescu_VanSchaftingen_CpctLift}; we give here a direct proof in all dimensions.

\begin{proof}[Proof of \cref{proposition_reverse_oscillation}]
\resetconstant
Since \(sp > 1\), we can fix \(\sigma \in \Rset\) so that \(\frac{1}{p} < \sigma  < s\).
There exists a constant \(\Cl{cst_ohshuNesh5wuchaeGh4akazo}\) such that for every \(x, y \in \Rset^n\), we have 
\begin{equation}
\label{eq_eu7oothe3aizeeChee8phu1j}
 \brk[\big]{\essosc_{[x, y]} u}^p  \le 
 \Cr{cst_ohshuNesh5wuchaeGh4akazo}\smashoperator{\iint\limits_{\intvc{0}{1} \times \intvc{0}{1}}} \frac{\abs{u ((1-t) x + t y) - u((1 - r) x + r y)}^p}{\abs{t - r}^{1 + \sigma p}} \dif t \dif r.
\end{equation}
Indeed, since \(\sigma p > 1\), by the fractional Morrey--Sobolev embedding there exists a constant \(\Cr{cst_ohshuNesh5wuchaeGh4akazo}\) such that (see \cite{DiNezza_Palatucci_Valdinoci_2012}*{\S 8}) for almost every \(\rho, \tau \in [0, 1]\),
\begin{multline} 
\label{eq_vohkaXu8fai3eesh5Soo7wah}
\abs[\big]{u \brk[\big]{(1 - \rho)x + \rho y} - u \brk[\big]{(1 - \tau)x + \tau y}}^p  \\
\le 
 \Cr{cst_ohshuNesh5wuchaeGh4akazo}\smashoperator{\iint\limits_{\intvc{0}{1} \times \intvc{0}{1}}} \frac{\abs{u ((1-t) x + t y) - u((1 - r) x + r y)}^p}{\abs{t - r}^{1 + \sigma p}} \dif t \dif r,
\end{multline}
and \eqref{eq_eu7oothe3aizeeChee8phu1j} follows from the definition of essential oscillation \eqref{eq_qui5eihee1rai9IengieTeeN} and from the estimate \eqref{eq_vohkaXu8fai3eesh5Soo7wah}.

Integrating \eqref{eq_eu7oothe3aizeeChee8phu1j} with respect to \(x, y \in \Omega\) we get 
 \begin{multline}
 \label{eq_sahThoo1aiK3gaeH5Ephah7a}
   \iint\limits_{\Omega \times \Omega}\frac{(\essosc_{[x, y]} u)^p}{\abs{y - x}^{n + sp}} \dif y \dif x\\
    \le \Cr{cst_ohshuNesh5wuchaeGh4akazo}
     \iint\limits_{\Omega \times \Omega}\smashoperator[r]{\iint\limits_{\intvc{0}{1} \times \intvc{0}{1}}} \frac{\abs{u ((1-t) x + t y) - u((1 - r) x + r y)}^p}{\abs{t - r}^{1 + \sigma p}\abs{y- x}^{n + sp}} \dif t \dif r \dif y \dif x.
\end{multline}
Applying in the right-hand side of \eqref{eq_eu7oothe3aizeeChee8phu1j} the change of variable \((x, y) \mapsto (w, z) = ((1-t) x + t y,  (1 - r) x + r y)\), we get 
  \begin{multline}
  \label{eq_VeepeG3choo8moh3eemiigho}
   \iint\limits_{\Omega \times \Omega}\frac{(\essosc_{[x, y]} u)^p}{\abs{y - x}^{n + sp}} \dif y \dif x
    \le \Cr{cst_ohshuNesh5wuchaeGh4akazo}
    \iint\limits_{\Omega \times \Omega}\iint\limits_{\Sigma_{z, w}} \frac{\abs{u (z)- u(w)}^p}{\abs{t - r}^{1  - (s - \sigma) p} \abs{z - w}^{n + sp}} \dif t \dif r \dif z \dif w.
  \end{multline}
  where for each \(z, w \in \Omega\) we have defined the set
  \begin{equation}
  \label{eq_oroothai0Iodiethaiph9Cah}
  \Sigma_{z, w} 
  \defeq
  \setcond*{
  (t, r) \in [0, 1]^2}{
  \tfrac{rz - tw}{r - t} \in \Omega 
  \text{ and } 
  \tfrac{(1 - r)z - (1 -t)w}{t - r} \in \Omega 
  }.
  \end{equation}
  We conclude by estimating the innermost integral in the right-hand side of \eqref{eq_oroothai0Iodiethaiph9Cah} by monotonicity of the integral as 
  \begin{equation}
  \label{eq_Eech8caid5uVe0Zaahi1owaN}
  \begin{split}
    \iint\limits_{\Sigma_{z, w}} \frac{1}{\abs{t - r}^{1 - (s - \sigma)p}} \dif t \dif r
    &\le \iint\limits_{\intvc{0}{1} \times \intvc{0}{1}} \frac{1}{\abs{t - r}^{1 - (s - \sigma)p}} \dif t \dif r\\
    & = \frac{1}{(s - \sigma)p} \int_{0}^1 \abs{1 - r}^{(s - \sigma)p} + \abs{r}^{(s - \sigma)p}\dif r < +\infty,
  \end{split}
  \end{equation}
  since \(\sigma > s\).
  The conclusion follows from \eqref{eq_VeepeG3choo8moh3eemiigho} and \eqref{eq_Eech8caid5uVe0Zaahi1owaN}.
\end{proof}

\begin{proof}%
[Proof of \cref{theorem_reverse_fractional}]%
\resetconstant
Since \(sp > 1\), for almost every \([x, y]\), by the fractional Morrey embedding, the closed set \(\essrg_{[x, y]} u \subset \essrg_{\Omega} u\) is compact and connected, and \(\essrg_{[x, y]} f \compose u = f (\essrg_{[x, y]} u)\), we have thus for almost every \(x, y \in \Omega\),
\begin{equation}
\label{eq_et1coh1eleebieghoo3ooTh3}
\begin{split}
  \abs{u (y) - u (x)} 
  \le \essosc_{[x, y]} u
  &= \diam (\essrg_{[x, y]} u)\\
  &\le \lambda 
  \diam (\essrg_{[x, y]} f\compose u)
  = \lambda\essosc_{[x, y]} f \compose u.
  \end{split}
\end{equation}
where 
\begin{equation}
 \lambda\defeq 
 \sup\, \setcond[\Big]{\frac{\diam (K)}{\diam (f(K))}}{K \subset \essrg u \text{ compact, connected and } \diam(K) > 0}.
\end{equation}
By \eqref{eq_et1coh1eleebieghoo3ooTh3} and the reverse oscillation inequality \cref{proposition_reverse_oscillation}, we conclude that there exists a constant \(C\) such that 
\begin{equation}
\iint\limits_{\Omega \times \Omega}\frac{\abs{u (y) - u (x)}^p}{\abs{y - x}^{n + sp}} \dif y \dif x
    \le C \lambda{}^p \iint\limits_{\Omega \times \Omega} \frac{\abs{f (u (y))- f(u(x))}^p}{\abs{y - x}^{n + sp}}  \dif y \dif x
    \eofs.
    \qedhere
\end{equation}
\end{proof}

\section{Counterexamples}

The following example shows that in the rough case \(sp < 1\), the fractional reverse estimate \cref{theorem_reverse_fractional} fails as soon as the function \(f\) is not injective.

\begin{proposition}
\label{proposition_non_estimate_sp_lt_1}
Let \(\Omega \subseteq \Rset^n\), \(s \in (0, 1)\) and \(p \in [1, +\infty)\).
If \(sp < 1\) and if the function \(f : \Rset^m \to \Rset^\ell\) is not injective, then there exists a sequence \((u_j)_{j \in \Nset}\) in \(\dot{W}^{s, p} (\Omega, \Rset^m)\) such that for every \(j \in \Nset\), the function \(f \compose u_j\) is constant \(\Omega\) and such that 
\begin{equation}
 \lim_{j \to \infty}
 \iint_{\Omega \times \Omega} \frac{\abs{u_j (y) - u_j (x)}^p}{\abs{y - x}^{n + sp}} \dif y \dif x = +\infty.
\end{equation}
\end{proposition}
\begin{proof}
We consider the case \(\Omega = (0, 1)\subset\Rset\); the other cases are similar.
By assumption, there exist two points \(b_0, b_1 \in \Rset^m\) such that \(f (b_0) = f (b_1)\). For each \(j \in \Nset\) we define  the function  \(u_j : (0, 1) \to \Rset^m\) for every \(x \in (0, 1)\) by 
\begin{equation}
u_j (x) \defeq
\begin{cases}
 b_0 &\text{if \(j x \in [2k, 2k + 1)\) for some \(k \in \Zset\),}\\
 b_1 &\text{if \(j x \in [2k + 1, 2(k +1))\) for some \(k \in \Zset\).}
\end{cases}
\end{equation}
By construction, for each \(j \in \Nset\), we have  \(f(u_j) = f(b_0)= f(b_1)\) everywhere in the interval \((0, 1)\).
Estimating 
\begin{equation}
\begin{split}
 \smashoperator{\iint_{(0, 1) \times (0, 1)}} \frac{\abs{u_j (y) - u_j (x)}^p}{\abs{y - x}^{1 + sp}} \dif y \dif x
 &= \sum_{\ell = 0}^{j - 1}\smashoperator[r]{\iint_{(\frac{\ell}{j}, \frac{\ell + 1}{j}) \times (0, 1)}}\frac{\abs{u_j (y) - u_j (x)}^p}{\abs{y - x}^{1 + sp}} \dif y \dif x\\
 &\le \sum_{\ell = 0}^{j - 1} 
 \smashoperator[r]{\iint_{(\frac{\ell}{j}, \frac{\ell + 1}{j}) \times \Rset\setminus (\frac{\ell}{j}, \frac{\ell + 1}{j})}}\frac{\abs{b_0 - b_1}^p}{\abs{y - x}^{1 + sp}} \dif y \dif x\\
 &= j^{sp} \smashoperator[l]{\iint_{(0, 1) \times \Rset\setminus (0, 1)}} \frac{\abs{b_0 - b_1}^p}{\abs{y - x}^{1 + sp}} \dif y \dif x = \frac{2 j^{sp}\abs{b_0 - b_1}^p}{sp(1 - sp)},
 \end{split}
\end{equation}
we infer that for each \(j \in \Nset\), we have \(u_j \in \dot{W}^{s, p} ((0, 1), \Rset^m)\).
Finally, we have if \(j \in \Nset_*\),
\begin{equation}
\begin{split}
 \smashoperator{\iint_{(0, 1) \times (0, 1)}} \frac{\abs{u_j (y) - u_j (x)}^p}{\abs{y - x}^{1 + sp}} \dif y \dif x
 &\ge \sum_{\ell = 1}^{j - 1}
 \smashoperator[r]{\iint_{(\frac{\ell - 1}{j}, \frac{\ell}{j}) \times (\frac{\ell}{j}, \frac{\ell + 1}{j})}}\frac{\abs{b_0 - b_1}^p}{\abs{y - x}^{1 + sp}} \dif y \dif x\\
 &\ge (1- \tfrac{1}{j})j^{sp} \smashoperator[l]{\iint_{(-1, 0)\times (0, 1)}} \frac{\abs{b_0 - b_1}^p}{\abs{y - x}^{1 + sp}} \dif y \dif x\\
  & = \frac{2(1- \tfrac{1}{j})j^{sp}(1- 2^{-sp})}{sp(1 - sp)}\abs{b_1 - b_0}^{p},
\end{split} 
\end{equation}
which goes to \(+\infty\) as \(j \to \infty\).
\end{proof}
Finally, in the critical case \(sp = 1\), the fractional reverse estimate of  \cref{theorem_reverse_fractional} fails when the function \(f\) is Lipschitz continuous and not injective.

\begin{proposition}
\label{proposition_non_estimate_sp_eq_1}
Let \(\Omega \subseteq \Rset^n\), \(s \in (0, 1)\) and \(p \in [1, +\infty)\).
If \(sp = 1\), if the function \(f : \Rset^m \to \Rset^\ell\) is Lipschitz-continuous and is not injective, then there exists a sequence \((u_j)_{j \in \Nset}\) in \(\dot{W}^{s, p} (\Omega, \Rset^m)\) such that  
\begin{equation}
 \lim_{j \to \infty}
 \smashoperator{\iint_{\Omega \times \Omega}} \frac{\abs{u_j (y) - u_j (x)}^p}{\abs{y - x}^{n + sp}} \dif y \dif x = +\infty
\end{equation}
and 
\begin{equation}
 \sup_{j \in \Nset}
 \smashoperator{\iint_{\Omega \times \Omega}} \frac{\abs{f \compose u_j (y) - f \compose u_j (x)}^p}{\abs{y - x}^{n + sp}} \dif y \dif x < +\infty.
\end{equation}
\end{proposition}
\begin{proof}
We concentrate on the case \(\Omega = (-1, 1)\subset \Rset\), the other cases being similar.
By our assumption, there are two points \(b_0, b_1 \in \Rset^m\) such that \(f (b_0) = f (b_1)\). 
We define the function \(u_* : \Rset \to \Rset^n\) for each  \(t \in \Rset\) by 
\begin{equation}
 u_* (t) \defeq 
 \begin{cases}
  b_0 & \text{if \(t \le - 1\)},\\
  \frac{1 - t}{2} b_0 + \frac{1+t}{2} b_1 & \text{if \(-1 < t < 1\)},\\
  b_1 & \text{if \(t \ge 1\)}.
 \end{cases}
\end{equation}
and we define for every \(j \in \Nset\) the function \(u_j : (-1, 1) \to \Rset^m\) by setting for each \(x \in (-1, 1)\), \(u_j (x) \defeq u(jx)\).
Since the function \(u_*\) is Lipschitz-continuous, we have \(u_j \in \dot{W}^{s, p} ((-1, 1), \Rset^m)\).
Since \(sp = 1\), we have for every \(j \in \Nset\),
\begin{equation}
\begin{split} \smashoperator{\iint_{(-1, 1) \times (-1, 1)}} \frac{\abs{f (u_j (y)) - f (u_j (x)})^p}{\abs{y - x}^{1 + sp}} \dif y \dif x
 &\le  \smashoperator{\iint_{\Rset \times \Rset}} \frac{\abs{f (u_* (j y)) - f (u_* (j x))}^p}{\abs{y - x}^{2}} \dif y \dif x\\
 &= \smashoperator{\iint_{\Rset \times \Rset}} \frac{\abs{f (u_* (y)) - f (u_* (x))}^p}{\abs{y - x}^{2}} \dif y \dif x <+\infty.
\end{split}
\end{equation}
On the other hand, we have for every \(j \in \Nset\)
\begin{equation}
  \smashoperator{\iint_{(-1, 1) \times (-1, 1)}} \frac{\abs{u_j (y) - u_j (x)}^p}{\abs{y - x}^{1 + sp}} \dif y \dif x
  \ge 2 \int_{-1}^{-\frac{1}{j}} \int_{\frac{1}{j}}^1
  \frac{\abs{b_1 - b_0}^p}{\abs{y - x}^2} \dif y \dif x
  = 2 \abs{b_1 - b_0}^p \ln \frac{(j + 1)^2}{4 j},
\end{equation}
which blows up as \(j \to \infty\).
\end{proof}


\begin{bibdiv}
\begin{biblist}

\bib{Ambrosio_DalMaso_1990}{article}{
   author={Ambrosio, L.},
   author={Dal Maso, G.},
   title={A general chain rule for distributional derivatives},
   journal={Proc. Amer. Math. Soc.},
   volume={108},
   date={1990},
   number={3},
   pages={691--702},
   issn={0002-9939},
   doi={10.2307/2047789},
}
\bib{DiNezza_Palatucci_Valdinoci_2012}{article}{
   author={Di Nezza, Eleonora},
   author={Palatucci, Giampiero},
   author={Valdinoci, Enrico},
   title={Hitchhiker's guide to the fractional Sobolev spaces},
   journal={Bull. Sci. Math.},
   volume={136},
   date={2012},
   number={5},
   pages={521--573},
   issn={0007-4497},
   review={\MR{2944369}},
   doi={10.1016/j.bulsci.2011.12.004},
}
\bib{Gilbarg_Trudinger_1983}{book}{
   author={Gilbarg, David},
   author={Trudinger, Neil S.},
   title={Elliptic partial differential equations of second order},
   series={Grundlehren der Mathematischen Wissenschaften},
   volume={224},
   edition={2},
   publisher={Springer}, 
   address={Berlin},
   date={1983},
   pages={xiii+513},
   isbn={3-540-13025-X},
   doi={10.1007/978-3-642-61798-0},
}

\bib{Leoni_2017}{book}{
      author={Leoni, Giovanni},
      title={A first course in Sobolev spaces},
      series={Graduate Studies in Mathematics},
      volume={181},
      edition={2},
      publisher={American Mathematical Society}, 
      address={Providence, R.I.},
      date={2017},
      pages={xxii+734},
      isbn={978-1-4704-2921-8},
    }

\bib{Marcus_Mizel_1972}{article}{
   author={Marcus, M.},
   author={Mizel, V. J.},
   title={Absolute continuity on tracks and mappings of Sobolev spaces},
   journal={Arch. Rational Mech. Anal.},
   volume={45},
   date={1972},
   pages={294--320},
   issn={0003-9527},
   doi={10.1007/BF00251378},
}
		
\bib{Mironescu_VanSchaftingen_CpctLift}{article}{
  author={Mironescu, Petru},
  author={Van Schaftingen, Jean},
  title={Lifting in compact covering spaces for fractional Sobolev mappings}
  journal={to appear in Anal. PDE},
  eprint={arXiv:1907.01373},
}


\bib{Musina_Nazarov_2019}{article}{
   author={Musina, Roberta},
   author={Nazarov, Alexander I.},
   title={A note on truncations in fractional Sobolev spaces},
   journal={Bull. Math. Sci.},
   volume={9},
   number={1},
   pages={1950001, 7},
   issn={1664-3607},
   doi={10.1142/S1664360719500012},
}

\bib{Willem_2013}{book}{
   author={Willem, Michel},
   title={Functional analysis},
   subtitle={Fundamentals and applications},
   publisher={Birkh\"{a}user/Springer, New York},
   date={2013},
   pages={xiv+213},
   isbn={978-1-4614-7003-8},
   isbn={978-1-4614-7004-5},
   doi={10.1007/978-1-4614-7004-5},
}

\bib{Ziemer_1989}{book}{
   author={Ziemer, William P.},
   title={Weakly differentiable functions},
   series={Graduate Texts in Mathematics},
   volume={120},
   subtitle={Sobolev spaces and functions of bounded variation},
   publisher={Springer}, 
   address={New York},
   date={1989},
   pages={xvi+308},
   isbn={0-387-97017-7},
   doi={10.1007/978-1-4612-1015-3},
}
		
\end{biblist}
\end{bibdiv}
\end{document}